\documentclass[12pt]{article}

\usepackage{amsmath,amssymb,amsthm,amsfonts,latexsym,hyperref,enumerate}
\usepackage{xcolor}
\usepackage{graphicx}
\usepackage{enumitem}

\def\N{{\mathbb N}}

\def\R{{\mathbb R}}

\newtheorem{lemma}{Lemma}[section]
\newtheorem{theorem}[lemma]{Theorem}
\newtheorem{corollary}[lemma]{Corollary}
\newtheorem{proposition}[lemma]{Proposition}

\title{Neumaier graphs with few eigenvalues}

\author{Aida Abiad\thanks{\texttt{a.abiad.monge@tue.nl},  Department of Mathematics and Computer Science, Eindhoven University of Technology, The Netherlands}\thanks{Department of Mathematics: Analysis, Logic and Discrete Mathematics, Ghent University, Belgium} \quad Bart De Bruyn\thanks{\texttt{Bart.DeBruyn@ugent.be}, Department of Mathematics: Algebra and Geometry, Ghent University, Belgium}  \quad Jozefien D'haeseleer\thanks{\texttt{Jozefien.Dhaeseleer@ugent.be}, Department of Mathematics: Analysis, Logic and Discrete Mathematics, Ghent University, Belgium} \quad Jack H. Koolen\thanks{\texttt{koolen@ustc.edu.cn}, School of Mathematical Sciences, University of Science and Technology of China, Hefei, China}\thanks{CAS Wu Wen-Tsun Key Laboratory of Mathematics, University of Science and Technology of China, Hefei, China}}

\date{}

\begin{document}

\maketitle


\begin{abstract}
A Neumaier graph is a non-complete edge-regular graph  containing a regular  clique. 
In this paper we give some sufficient and necessary conditions for a Neumaier graph to be strongly regular. Further we show that there does not exist Neumaier graphs with exactly four distinct eigenvalues. We also determine the Neumaier graphs with smallest eigenvalue $-2$.

\end{abstract}


\section{Introduction}

A regular graph is called \emph{edge-regular} if any two adjacent vertices have the same number of common neighbours. A \emph{regular clique} in a regular graph is a clique $C$ having the property that every vertex outside $C$ is adjacent to the same positive number of vertices of $C$. In the early 1980s, Neumaier \cite{Neumaier1981} studied regular cliques in edge-regular graphs, and a certain class of designs whose point graphs are strongly regular and contain regular cliques. He then posed the problem of whether there exists a non-complete, edge-regular, non-strongly regular graph containing a regular clique, [\cite{Neumaier1981}, page 248]. 

We define a \emph{Neumaier graph} as a non-complete edge-regular graph containing a regular clique. A Neumaier graph that is not a strongly regular graph is called a \emph{strictly Neumaier graph}. 

Neumaier graphs have received quite a lot of attention recently. Greaves and  Koolen  \cite{GK2018} provided the first infinite family of  Neumaier graphs, using cyclotomic numbers, showing that strictly Neumaier graphs exist. Goryainov and Shalaginov \cite{GoSha} classified  the Cayley-Deza graphs with fewer than 60 vertices, and it turned out that there are four strictly Neumaier graphs with 20 vertices among them.
Greaves and Koolen \cite{GK2019} found a new infinite family of strictly Neumaier graphs containing some of the examples found by Goryainov and Shalaginov.

Evans, Goryainov and Panasenko \cite{EGP2019} obtained some new infinite families of strictly Neumaier graphs by switching an affine polar graph over GF$(2^n)$ where $n$ is a positive integer.
They also showed some general results on Neumaier graphs and their feasible parameter tuples, and presented an application of such results to determine the smallest non-strongly regular Neumaier graph, answering some questions posted by Greaves and Koolen in \cite{GK2018}.

We continue the study of Neumaier graphs by studying some of their spectral properties. In particular, in this paper we take a closer look at Neumaier graphs with few distinct eigenvalues. In Section \ref{sec:Neumaiersrg} we present several combinatorial and spectral conditions under which a Neumaier graph is a strongly regular graph. We also determine the Neumaier graphs with smallest eigenvalue $-2$. In Section \ref{Neumaier4distinctev} we show some feasibility conditions of Neumaier graphs with four distinct eigenvalues and we use them to prove that there does  not exist strictly Neumaier graphs with exactly four distinct eigenvalues.

\section{Strongly regular Neumaier graphs}\label{sec:Neumaiersrg}

In this section we study under which conditions Neumaier graphs have exactly three distinct eigenvalues, that is, we provide several new characterizations of strongly regular Neumaier graphs. A result in this direction was already obtained by Neumaier in Corollary 2.4 of \cite{Neumaier1981}. Recall that not all Neumaier graphs are strongly regular graphs, as it was shown with the constructions by Greaves and Koolen \cite{GK2018,GK2019}, and Evans, Goryainov and Panasenko \cite{EGP2019}.

\subsection{A combinatorial characterization}\label{Sec:Neumaier1file}

Let $\Gamma=(V,E)$ be a non-complete edge-regular graph on $v > 0$ vertices having valency $k \geq 1$. We denote by $\lambda$ the constant number of triangles through a given edge. If $\Gamma$ is a complete multipartite graph, then we denote the total number of its multipartite parts by $s+1$. If $\Gamma$ is not a complete multipartite graph, then we show in Lemma \ref{lem1} that $v + \lambda -2k \not= 0$. In this case, we define $s$ to be the following number:
\[ s := \frac{-(k^2-k+\lambda-v \lambda) + \sqrt{(k^2-k+\lambda-v\lambda)^2 + 4k(v-k-1)(v+\lambda-2k)}}{2(v+\lambda-2k)}.  \]
In Propositions \ref{prop2} and \ref{prop3} below, we show the following.
\begin{itemize}
\item Every clique $C$ of $\Gamma$ has order at most $s+1$, with equality if and only if $C$ is a regular clique.
\item If $\Gamma$ has a regular clique (of order $s+1$), then $v \not= s+1$ (as $\Gamma$ is not complete) and every vertex outside $C$ is adjacent to precisely $e:=\frac{(s+1)(k-s)}{v-(s+1)}$ vertices of $C$.
\end{itemize}
Note that if $\Gamma$ is a Neumaier graph, then there are (regular) cliques of order $s+1$. The main goal of this subsection is to prove the following characterization of a particular family of strongly regular Neumaier graphs. 

\begin{theorem}\label{mainresultNeumaier1file}
If $\Gamma$ is a Neumaier graph with the property that every clique of order $e+1$ is contained in a clique of order $s+1$, then $\Gamma$ is strongly regular.
\end{theorem}

\medskip \noindent We note that there exist infinite families of Neumaier graphs that have the property mentioned in Theorem \ref{mainresultNeumaier1file}. These include the Neumaier graphs that are complete multipartite graphs and the collinearity graphs of finite generalized quadrangles or finite polar spaces.  

\medskip \noindent We will denote the neighbourhood of a vertex $v$ of $\Gamma$ by $v^\perp = \Gamma(v)$. If $v$ is a vertex, then $\Gamma_i(v)$ with $i \in \N$ denotes the set of vertices at distance $i$ from $v$. If $v$ and $w$ are two vertices of $\Gamma$, then we write $v \sim w$ or $v \not\sim w$ depending on whether $v$ and $w$ are adjacent or not. We denote the edge through two adjacent vertices $v$ and $w$ by $vw$. For any set $X$ of vertices of $\Gamma$, we denote by $\overline{X} := V \setminus X$ its complement in $V$.

\medskip \noindent If $\Gamma$ is a complete multipartite graph, then the fact that $k \geq 1$ implies that the number $(s+1)$ of multipartite parts is at least 2, and the fact that $\Gamma$ is regular and non-complete implies that all multipartite parts have the same order $m \geq 2$. In this case, we have $v = (s+1)m$, $k = sm$, $\lambda = (s-1)m$ and hence $v+\lambda-2k=0$. The latter property is sufficient to characterize complete multipartite graphs as the following lemma shows.

\begin{lemma} \label{lem1}
We have $v+\lambda-2k \geq 0$, with equality if and only if $\Gamma$ is a complete multipartite graph.
\end{lemma}
\begin{proof}
If $xy$ is an edge of $\Gamma$, then $v \geq |x^\perp \cup y^\perp| = 2k-\lambda$, implying that $v+\lambda - 2k \geq 0$. Suppose now that equality holds. Then for every three distinct vertices $x$, $y$ and $z$ for which $x \sim y$, we have that $z$ is adjacent to at least one of $x$, $y$. If $u_1$ and $u_2$ are two non-adjacent vertices and $u_1w$ is an edge, then the fact that $u_1 \not\sim u_2$ implies that $w \sim u_2$. 
So, the neighbourhoods of two vertices of $\Gamma$ are equal if and only if these vertices are non-adjacent. The non-adjacency relation on the vertex set is thus an equivalence relation and $\Gamma$ is a complete multipartite graph.
\end{proof}

\medskip \noindent Neumaier already observed that if an edge-regular graph has a regular clique, say of order $c$, then any regular clique has order $c$ and any clique with order $c$ is regular [\cite{Neumaier1981}, Theorem 1.1]. The following is implicitly shown in \cite{Neumaier1981}, but we add its proof for completeness.

\begin{proposition} \label{prop2}
Suppose $\Gamma$ is not a complete multipartite graph. Every clique $C$ of $\Gamma$ then has order at most $s+1$, and equality occurs if and only if $C$ is a regular clique. In this case, every vertex outside $C$ is adjacent to precisely $e$ vertices of $C$.
\end{proposition}
\begin{proof}
Suppose $C$ is a clique of order $\sigma+1 \geq 1$. For every $x \in \overline{C}$, let $e_x$ denote the number of vertices in $C$ adjacent to $x$. Then
\begin{equation}
\sum_{x \in \overline{C}} 1 = v-(\sigma+1).
\end{equation}
Counting in two different ways the pairs $(y,x) \in C \times \overline{C}$ with $y \sim x$ gives
\begin{equation} \label{eq22}
\sum_{x \in \overline{C}} e_x = (\sigma+1)(k-\sigma).
\end{equation}
Counting in two different ways the pairs $(y_1,y_2,x) \in C \times C \times \overline{C}$ with $y_1 \not= y_2$ and $y_1 \sim x \sim y_2$ gives
\begin{equation} \label{eq32}
\sum_{x \in \overline{C}} e_x(e_x-1) = (\sigma+1)\sigma(\lambda-(\sigma-1)).
\end{equation}
From (\ref{eq22}) and (\ref{eq32}), we find
\begin{equation} \label{eq4}
\sum_{x \in \overline{C}} e_x^2 = (\sigma+1)(\sigma \lambda - \sigma^2 +k).
\end{equation}
From the Cauchy-Schwartz inequality, we know that
\[  \Big( \sum e_x \Big)^2 \leq \Big( \sum 1 \Big) \cdot \Big( \sum e_x^2  \Big),   \]
with equality if and only if all $e_x$'s are equal, i.e., if and only if $\{ C,\overline{C} \}$ is an equitable partition. We thus find
\[ (\sigma+1)^2(k-\sigma)^2 \leq (v-(\sigma+1))(\sigma+1)(k+\sigma \lambda-\sigma^2), \]
\[ (\sigma+1)(k^2-2\sigma k+\sigma^2) \leq (v-\sigma-1)(k+\sigma \lambda-\sigma^2), \]
\[ k^2\sigma -2\sigma^2k+\sigma^3+k^2-2\sigma k+\sigma^2 \leq vk+v\lambda \sigma-v\sigma^2 - \sigma k - \sigma^2\lambda +\sigma^3 -k - \sigma \lambda +\sigma^2, \]
i.e.,
\begin{equation} \label{eq52}
(v+\lambda-2k)\sigma^2 + (k^2-k+\lambda-v\lambda)\sigma+(k^2+k-vk) \leq 0.
\end{equation}
By Lemma \ref{lem1}, we know that $v+\lambda-2k > 0$. As $k \geq 1$ and $\Gamma$ is not complete, we also have  $k^2+k-vk=-k(v-(k+1)) < 0$. So, the quadratic equation (\ref{eq52}) in the variable $\sigma$ has two real roots $s_1$ and $s_2$ with $s_1 < 0$ and $s_2>0$. A straightforward computation shows that $s_2=s$. So, (\ref{eq52}) implies that $\sigma \leq s$. By the above we also know that $\sigma=s$ if and only if all $e_x$'s are equal, i.e., if and only if $C$ is a regular clique. If all $e_x$'s are equal, then they are equal to their average value $(\sum e_x) \cdot (\sum 1)^{-1} = \frac{(\sigma+1)(k-\sigma)}{v-\sigma-1}=\frac{(s+1)(k-s)}{v-s-1}=e$.
\end{proof}

\medskip \noindent A similar property as in Proposition \ref{prop2} holds in case $\Gamma$ is a complete multipartite graph. Note that if $\Gamma$ is a complete multipartite graph having $s+1 \geq 2$ parts of size $m \geq 2$, then $e = \frac{(s+1)(k-s)}{v-(s+1)} = \frac{(s+1)(ms-s)}{(s+1)m-(s+1)}=s$. 

\begin{proposition} \label{prop3}
Suppose $\Gamma$ is a complete multipartite graph. Every clique $C$ of $\Gamma$ then has order at most $s+1$, and equality occurs if and only if $C$ is a regular clique. In this case, every vertex outside $C$ is adjacent to precisely $e=s$ vertices of $C$.
\end{proposition}
\begin{proof}
Recall that $\Gamma$ has exactly $s+1$ multipartite parts. The obvious bound for $|C|$ is therefore $s+1$. If $|C|=s+1$, then $C$ contains exactly one point of each of the $s+1$ multipartite parts. In this case, every point outside $C$ is adjacent to exactly $s$ vertices of $C$. If $0 < |C| < s+1$, then there is a point outside $C$ that is adjacent to all $|C|$ vertices of $C$ and a point outside $C$ that is adjacent to exactly $|C|-1$ vertices of $C$, proving that $C$ cannot be a regular clique.
\end{proof}

\medskip \noindent Our next aim is to prove Theorem \ref{mainresultNeumaier1file}. So, from now on, we assume that $\Gamma$ is a Neumaier graph with the property that every clique of order $e+1$ is contained in a clique of order $s+1$. Our intention is thus to prove that $\Gamma$ is strongly regular. In the sequel, a clique of order $l$ will shortly be called an {\em $l$-clique}. By Propositions \ref{prop2} and \ref{prop3}, we have:

\begin{lemma} \label{lem4}
Every $(s+1)$-clique is a maximum clique. Hence, $e < s+1$.
\end{lemma}

\begin{lemma} \label{lem5}
Every $(e+1)$-clique $H$ is contained in a unique $(s+1)$-clique.
\end{lemma}
\begin{proof} 
This immediately follows from Lemma $1.5(i)$ in \cite{Neumaier1981}. If $C_1$ and $C_2$ were two distinct $(s+1)$-cliques containing $H$, then every vertex of $C_2 \setminus C_1$ is adjacent to all $|C_1 \cap C_2| \geq e+1$ vertices of $C_1 \cap C_2$, an obvious contradiction.
\end{proof}

\begin{lemma} \label{lem6}
Every clique $H$ is contained in an $(s+1)$-clique.
\end{lemma}
\begin{proof}
Let $C'$ denote an $(s+1)$-clique for which $|H \cap C'|$ is maximal. We show that $H \subseteq C'$. If this were not the case, then there exists a vertex $x \in H \setminus C'$. The set $x^\perp \cap C'$ then contains $H \cap C'$ and so the unique $(s+1)$-clique $C''$ containing the $(e+1)$-clique $\{ x \} \cup (x^\perp \cap C')$ would have a larger intersection with $H$ than $C'$.
\end{proof}

\begin{lemma} \label{lem7}
The graph $\Gamma$ has diameter $2$.
\end{lemma}
\begin{proof}
Since $\Gamma$ is not a complete graph, it suffices to prove that any two distinct non-adjacent vertices $x$ and $y$ have distance $2$. By Lemma \ref{lem6}, there exists an $(s+1)$-clique $C$ through $x$. Then $y \not\in C$ and $y$ is adjacent to $e \geq 1$ vertices of $C$. So, $x$ and $y$ lie at distance 2 from each other.
\end{proof}

\begin{lemma} \label{lem8}
If $E_1$ and $E_2$ are two edges for which $|E_1 \cap E_2|=1$ and $E_1 \cup E_2$ is not a clique, then they are contained in the same number of $(s+1)$-cliques.
\end{lemma}
\begin{proof}
Put $E_1 = xy_1$, $E_2 = xy_2$ and let $i \in \{ 1,2 \}$. For every $(s+1)$-clique $C_i$ through $E_i$, there exists a unique $(s+1)$-clique $C_{3-i}$ through $E_{3-i}$ intersecting $C_i$ in exactly $e$ vertices. Indeed, this clique necessarily coincides with the unique $(s+1)$-clique containing the $(e+1)$-clique $\{ y_{3-i} \} \cup (y_{3-i}^\perp \cap C_i)$. The lemma then follows by counting in two ways the pairs $(C_1,C_2)$, where  each $C_i$, $i \in \{ 1,2 \}$, is an $(s+1)$-clique through $E_i$ and $|C_1 \cap C_2| = e$.
\end{proof}

\begin{lemma} \label{lem9}
If $E_1$ and $E_2$ are two edges for which $|E_1 \cap E_2|=1$ and $E_1 \cup E_2$ is a clique, then they are contained in the same number of $(s+1)$-cliques.
\end{lemma}
\begin{proof}
For any two adjacent vertices $u$ and $v$, let $N_{uv}$ denote the number of $s+1$-cliques containing $\{ u,v \}$. Put $E_1 = xy_1$ and $E_2 = xy_2$. Let $\eta$ denote the number of vertices adjacent to $x$, $y_1$ and $y_2$. Then for every $i \in \{ 1,2 \}$, the set $A_i$ of vertices distinct from $y_{3-i}$ adjacent to $x$ and $y_i$ but not to $y_{3-i}$ has order $\lambda - \eta - 1$.

Every $(s+1)$-clique through $E_i$ not containing $y_{3-i}$ contains $s+1-e > 0$ elements of $A_i$, namely the $s+1-e$ vertices of $S_i \setminus y_{3-i}^\perp$. So, if $\lambda-\eta-1=0$, then both $N_{xy_1}$ and $N_{xy_2}$ are equal to the number of $(s+1)$-cliques through $\{ x,y_1,y_2 \}$. In the sequel, we may therefore assume that $|A_1|=|A_2|=\lambda-\eta-1>0$.

For every $z_1 \in A_1$, we show that the number of $(s+1)$-cliques containing $\{ x,y_1,z_1 \}$ equals the number of $(s+1)$-cliques containing $\{ x,y_1,y_2 \}$. Put $T_1 = \{ x,y_1,z_1 \}$, $T_2 = \{ x,y_1,y_2 \}$, $u_1 = z_1$ and $u_2=y_2$. Note that $u_1 \not\sim u_2$. For every $i \in \{ 1,2 \}$ and every clique $C_i$ of order $s+1$ through $T_i$, there then exists a unique clique $C_{3-i}$ of order $s+1$ through $T_{3-i}$ intersecting $C_i$ in a set of order $e$. This clique $C_{3-i}$ is precisely the unique $(s+1)$-clique containing the $(e+1)$-clique $u_{3-i} \cup (u_{3-i}^\perp \cap C_i)$. The claim that the numbers of $(s+1)$-cliques coincide then follows from counting in two ways all pairs $(C_1,C_2)$ satisfying $T_1 \subseteq C_1$, $T_2 \subseteq C_2$ and $|C_1 \cap C_2| = e$.

In a completely similar way, one proves that if $z_2 \in A_2$, then the number of $(s+1)$-cliques containing $\{ x,y_2,z_2 \}$ equals the number of $(s+1)$-cliques containing $\{ x,y_1,y_2 \}$. If $z_1 \in A_1$ and $z_2 \in A_2$, we thus see that the number of $(s+1)$-cliques containing $\{ x,y_1,z_1 \}$ coincides with the number of $(s+1)$-cliques containing $\{ x,y_2,z_2 \}$. We call this number $\Omega$.

We count the number of $(s+1)$-cliques containing $E_1$, but not $E_2$. If $C$ is such a clique, then $C$ contains $s+1-e > 0$ vertices of $A_1$. As every vertex of $A_1$ is contained in $\Omega$ $(s+1)$-cliques together with $E_1$, and no such $(s+1)$-clique contains $E_2$, we see that the number of $(s+1)$-cliques containing $E_1$ but not $E_2$ is equal to
\[  \frac{|A_1| \cdot \Omega}{s+1-e} = \frac{(\lambda - \eta - 1) \Omega}{s+1-e}.  \]
In a similar way one proves that the total number of $(s+1)$-cliques containing $E_2$ but not $E_1$ equals $\frac{(\lambda - \eta-1) \Omega}{s+1-e}$. To find the total number of $(s+1)$-cliques through $E_i$, $i \in \{ 1,2 \}$, we should still add to this number the total number of $(s+1)$-cliques containing $\{ x,y_1,y_2 \}$. But this extra contribution is constant for both $E_1$ and $E_2$.
\end{proof}

\medskip \noindent By Lemmas \ref{lem8}, \ref{lem9} and the fact that $\Gamma$ is connected, we have:

\begin{corollary} \label{co10}
The number of $(s+1)$-cliques through an edge is constant.
\end{corollary}

\medskip \noindent As $\Gamma$ is regular of valency $k \geq 1$, this immediately implies the following via a double counting.

\begin{corollary} \label{co11}
The number of $(s+1)$-cliques through a vertex is constant.
\end{corollary}

\medskip \noindent By Lemma \ref{lem7}, Corollaries \ref{co10}, \ref{co11} and Proposition A2 of \cite{bdb-Suz}, we obtain the desired result which proves Theorem \ref{mainresultNeumaier1file}:

\begin{corollary} \label{co12}
The graph $\Gamma$ is strongly regular.
\end{corollary}

\subsection{Characterizations of strongly regular graphs using eigenvalues}

In this section we present some new eigenvalue conditions under which a connected and regular graph is strongly regular. We thus relax the conditions of requiring edge-regularity and a regular clique (i.e., a Neumaier graph) from Section \ref{Sec:Neumaier1file}. 

First we need to introduce some definitions. Let $\Gamma = (V,E)$ be a graph on $v \geq 3$ vertices that is not a complete multipartite graph, in particular, $\Gamma$ is not a complete graph nor a graph without edges.

We denote by $\bar{k} = \frac{2 \cdot |E|}{v} > 0$ the average degree of the vertices of $\Gamma$. We put $\bar{\lambda} := \frac{6N}{v\bar{k}}$, where $N$ is the total number of triangles of $\Gamma$. If we define $\lambda_{xy} := |\Gamma_1(x) \cap \Gamma_1(y)|$ for every edge $xy$ of $\Gamma$, then as $\sum_{xy \in E} \lambda_{xy} = 3N$ and $|E| = \frac{v\bar{k}}{2}$, we see that $\bar{\lambda}$ is the average of the $\lambda_{xy}$'s with $xy \in E$.

In Lemmas \ref{lem3} and \ref{lem4bis} below, we show that $v > \bar{k}+1$ and $v + \bar{\lambda} - 2\bar{k} > 0$. This implies that the quadratic polynomial
\[  (v+\bar{\lambda}-2\bar{k}) X^2 + (\bar{k}^2-\bar{k}+\bar{\lambda}-\bar{\lambda} v) X - \bar{k}(v-\bar{k}-1) \in \R[X] \]
has two real roots, a negative one and a positive one which we denote by $\bar{s}$. We thus have
\[ \bar{s} = \frac{-(\bar{k}^2-\bar{k}+\bar{\lambda}-\bar{\lambda} v) + \sqrt{(\bar{k}^2-\bar{k}+\bar{\lambda}-\bar{\lambda} v)^2+4 \bar{k}(v-\bar{k}-1)(v+\bar{\lambda}-2\bar{k})}}{2(v+\bar{\lambda}-2\bar{k})} > 0. \]   
We also define 
\[  \bar{\mu} := \frac{\bar{k}(\bar{k}-\bar{\lambda}-1)}{v-\bar{k}-1},\qquad \qquad \theta_m := -\frac{\bar k}{\bar s}, \qquad \qquad \theta_M :=\frac{\bar k-\bar \mu}{\bar k}\bar s. \]
If $\Gamma$ is a strongly regular graph with parameters $(v,k,\lambda,\mu)$, then $(\bar{k},\bar{\lambda},\bar{\mu})=(k,\lambda,\mu)$ and we will see that $k$, $\theta_M$ and $\theta_m$ are the eigenvalues of $\Gamma$ (see Lemma \ref{lem7bis} and its ensuing remark). If $\Gamma$ is a general graph, then we can interpret $\theta_M$ and $\theta_m$ by their algebraic definitions, depending on $\bar k$, $\bar s$ and $\bar \mu$. 

\medskip \noindent Our main result in this section is as follows.

\begin{theorem}\label{Neumaier4mainthm}
Suppose $\Gamma$ is a connected regular graph that is not a complete multipartite graph. If $\theta_{Max2}$ is the second largest eigenvalue and $\theta_{min}$ the smallest eigenvalue of $\Gamma$, then
\begin{description}
\item [$(a)$]  $\theta_{min} \leq \theta_m$ with equality if and only if $\Gamma$ is strongly regular, and
\item [$(b)$] $\theta_{Max2} \geq \theta_M$, and equality holds if and only if $\Gamma$ is strongly regular.
\end{description}
\end{theorem}

\bigskip \noindent In what follows we prove Theorem \ref{Neumaier4mainthm} as a special case of a more general result on arbitrary (not necessarily regular) graphs. First we recall some known facts about the spectrum of a graph.

\begin{lemma} \label{lem1bis}
\begin{itemize}
\item[$(a)$] If $\Gamma'$ is a graph having precisely one eigenvalue, then this eigenvalue is $0$ and $\Gamma'$ has no edges.
\item[$(b)$] If $\Gamma'$ is a graph having precisely two eigenvalues, then the connected components of $\Gamma'$ are complete graphs of the same order $t \geq 2$.
\item[$(c)$] Every connected regular graph having precisely three eigenvalues is strongly regular.
\end{itemize}
\end{lemma}

\medskip \noindent We denote by $\Omega$ the multiset of order $v$ whose elements are the eigenvalues of $\Gamma$ (taking into account their multiplicities). For every $\theta \in \Omega$, we denote by $\Omega_\theta$ the multiset of order $v-1$ obtained from $\Omega$ by removing one copy of $\theta$.

\begin{lemma} \label{lem2}
We have
\[  \sum_{\theta \in \Omega} \theta = 0,\qquad \sum_{\theta \in \Omega} \theta^2 = v\bar{k},\qquad \sum_{\theta \in \Omega} \theta^3 = v\bar{k}\bar{\lambda}.  \]
\end{lemma}
\begin{proof}
Put $S_1 :=  \sum_{\theta \in \Omega} \theta$, $S_2 := \sum_2 \theta \theta'$ and $S_3 := \sum_3 \theta \theta' \theta''$, with $\Sigma_2$ the summation over all $\{ \theta,\theta' \} \in {\Omega \choose 2}$ and $\Sigma_3$ the summation over all $\{ \theta,\theta',\theta'' \} \in {\Omega \choose 3 }$. If $x^n + c_1 x^{n-1} + c_2  x^{n-2} + \cdots + c_n = \prod_{\theta \in \Omega} (x-\theta)$ is the characteristic polynomial of $\Gamma$, then $S_1 = -c_1 =0$, $S_2=c_2$ and $S_3=-c_3$. As the number of edges is $\frac{v\bar{k}}{2}$ and the number of triangles is $\frac{v\bar{k}\bar{\lambda}}{6}$, we have $c_2=-\frac{v\bar{k}}{2}$ and $c_3=-\frac{v\bar{k}\bar{\lambda}}{3}$. So, $S_2 = -\frac{v\bar{k}}{2}$, $S_3 = \frac{v\bar{k}\bar{\lambda}}{3}$ and $\sum_{\theta \in \Omega} \theta^2 = S_1^2 -2S_2=v\bar{k}$ and $\sum_{\theta \in \Omega} \theta^3 =S_1^3 -3S_1S_2+3S_3 = v\bar{k}\bar{\lambda}$.
\end{proof}

\begin{lemma} \label{lem3}
We have $v > \bar{k}+1$.
\end{lemma}
\begin{proof}
For every vertex $x$ of $\Gamma$, we have $1 + |\Gamma_1(x)| = | \{ x \} \cup \Gamma_1(x) | \leq |V| = v$. Summing over all vertices $x$ of $\Gamma$, we find $v + v \cdot \bar{k} \leq v^2$, i.e. $v \geq \bar{k}+1$. If equality occurs, then $\{ x \} \cup \Gamma_1(x) = V$ for every vertex $x$. This is only possible when $\Gamma$ is a complete graph, contrary to our assumption. So, $v > \bar{k}+1$.
\end{proof}

\begin{lemma} \label{lem4bis}
We have $v + \bar{\lambda} - 2\bar{k} > 0$.
\end{lemma}
\begin{proof}
Put $k_x := |\Gamma_1(x)|$ for every vertex $x$ of $\Gamma$. For every edge $xy$ of $\Gamma$, we have $v = |V| \geq |\Gamma_1(x) \cup \Gamma_1(y)| = k_x + k_y - \lambda_{xy}$. Summing over all edges $xy$ of $\Gamma$, we find $\sum_{x \in V} k_x^2 - |E| \cdot \bar{\lambda} \leq v \cdot |E|$,
i.e.
\begin{equation} \label{eq11}
\frac{2}{v\bar{k}} \sum_{x \in V} k_x^2 - \bar{\lambda} \leq v.
\end{equation}
Now, by the Cauchy-Schwartz inequality, we have
\begin{equation} \label{eq11bis}
v^2 \bar{k}^2 = \Big( \sum_{x \in V} k_x \Big)^2 \leq \Big( \sum_{x \in V} 1 \Big) \cdot \Big( \sum_{x \in V} k_x^2 \Big) = v \cdot \sum_{x \in V} k_x^2,
\end{equation}
i.e.
\begin{equation} \label{eq21}
v \bar{k}^2 \leq \sum_{x \in V} k_x^2.
\end{equation}
By (\ref{eq11}) and (\ref{eq21}), we thus have $2\bar{k} - \bar{\lambda} \leq v$, i.e. $v + \bar{\lambda} - 2\bar{k} \geq 0$. In case of equality, so if $v + \bar{\lambda} - 2\bar{k}=0$, we know by the above that
\begin{enumerate}
\item[(a)] $\Gamma_1(x) \cup \Gamma_1(y) = V$ for every edge $xy$ of $\Gamma$;
\item[(b)] all $k_x$'s are equal, necessarily to $k:=\bar{k}$ (as equality holds in (\ref{eq11bis})).
\end{enumerate}
Condition (a) implies that $k_x + k_y - \lambda_{xy} = v$ for every edge $xy$. As $k_x = k_y = k$ and $v+\bar{\lambda}- 2k=0$, we thus have that $\lambda_{xy}=\bar{\lambda}$ for every edge $xy$ of $\Gamma$. The graph $\Gamma$ is thus an edge-regular graph. As in Lemma \ref{lem1}, we can then deduce that $\Gamma$ is a complete multipartite graph, contrary to our assumption. So, we have $v + \bar{\lambda} - 2\bar{k} > 0$.
\end{proof}

\begin{lemma} \label{lem5bis}
We have $(\bar{s}v-\bar{k}(\bar{s}+1))(v-\bar{k}-1)=(v-\bar{s}-1)(v+\bar{\lambda}-2\bar{k})\bar{s}$.
\end{lemma}
\begin{proof}
This is equivalent with $(v+\bar{\lambda}-2\bar{k}) \bar{s}^2 + (\bar{k}^2-\bar{k}+\bar{\lambda}-\bar{\lambda} v) \bar{s}-\bar{k} (v-\bar{k}-1)=0$.
\end{proof}

\begin{lemma} \label{lem6bis}
We have $v \not= \bar{s}+1$.
\end{lemma}
\begin{proof}
If $v=\bar{s}+1$, then Lemmas \ref{lem3} and \ref{lem5bis} imply that $\bar{k}=\bar{s}$. So, $v=\bar{k}+1$, which is in contradiction with Lemma \ref{lem3}.
\end{proof}

\medskip \noindent In view of Lemma \ref{lem6bis}, we can define the following number:
\[  \bar e :=\frac{(\bar s+1)(\bar k-\bar s)}{v-\bar s-1}.  \]

\begin{lemma} \label{lem7bis}
We have $X^2+(\bar{\mu}-\bar{\lambda})X+(\bar{\mu}-\bar{k}) = (X-\theta_m)(X-\theta_M)$.
\end{lemma}
\begin{proof}
As $\theta_m \theta_M = \bar{\mu}-\bar{k}$, it suffices to prove that $\theta_m=-\frac{\bar{k}}{\bar{s}}$ is a root of  $X^2+(\bar{\mu}-\bar{\lambda})X+(\bar{\mu}-\bar{k})$, or equivalently that
\[ \bar{k}^2-(\bar{\mu}-\bar{\lambda})\bar{k}\bar{s} + (\bar{\mu}-\bar{k})\bar{s}^2=0.  \]
As $\bar{\mu} = \frac{\bar{k}(\bar{k}-\bar{\lambda}-1)}{v-\bar{k}-1}$, we thus need to verify that
\[  \bar{k}^2(v-\bar{k}-1)-(\bar{k}(\bar{k}-\bar{\lambda}-1)-\bar{\lambda}(v-\bar{k}-1))\bar{k}\bar{s}+(\bar{k}(\bar{k}-\bar{\lambda}-1)-\bar{k}(v-\bar{k}-1))\bar{s}^2=0. \]
After division by $-\bar{k}$, the latter equation is equivalent with $(v+\bar{\lambda}-2\bar{k})\bar{s}^2+(\bar{k}^2-\bar{k}+\bar{\lambda}-\bar{\lambda} v)\bar{s} -\bar{k}(v-\bar{k}-1)=0$, from which we indeed know that that this is correct.
\end{proof}

\medskip \noindent If $\Gamma$ is a strongly regular graph with parameters $(v,k,\lambda,\mu)$, then we know that the eigenvalues of $\Gamma$ distinct from $k$ are the roots of $X^2+(\mu-\lambda)X+(\mu-k) \in \R[X]$, i.e. they are equal to $\theta_m$ and $\theta_M$. 

\begin{lemma}
We have $\theta_m < 0 < \theta_M$.
\end{lemma}
\begin{proof}
Obviously, $\theta_m = -\frac{\bar{k}}{\bar{s}} < 0$ and $\theta_M = \frac{\bar{k}-\bar{\mu}}{\bar{k}} \bar{s} = \frac{\bar{s}(v+\bar{\lambda}-2\bar{k})}{v-\bar{k}-1} > 0$ by Lemmas \ref{lem3} and \ref{lem4bis}.
\end{proof}

\begin{lemma} \label{lem8bis}
We have $\theta_M=\bar s-\bar e$.
\end{lemma}
\begin{proof}
We compute that $\bar s-\bar e = \bar s - \frac{(\bar s+1)(\bar k-\bar s)}{v-\bar s-1} = \frac{\bar s v-\bar k(\bar s+1)}{v-\bar s-1}$. In order for this to be equal to $\theta_M = \frac{\bar k-\bar \mu}{\bar k}\bar s = \frac{(v+\bar\lambda-2\bar k)\bar s}{v-\bar k-1}$, we must have that $(\bar s v-\bar k(\bar s+1))(v-\bar k-1)=(v-\bar s-1)(v+\bar \lambda-2\bar k)\bar s$. By that is precisely Lemma \ref{lem5bis}.
\end{proof}

\begin{lemma} \label{lem9bis}
We have $\bar k > \theta_M$, $\bar k=\theta_M$ or $\bar k < \theta_M$, depending on whether $\bar k > \bar\lambda+1$, $\bar k=\bar\lambda+1$ or $\bar k < \bar\lambda +1$, or equivalently, depending on whether $\bar\mu > 0$, $\bar\mu=0$ or $\bar\mu < 0$.
\end{lemma}
\begin{proof}
Taking into account Lemma \ref{lem7bis} and the definition of $\bar\mu$, we see that $(\bar k-\theta_m)(\bar k-\theta_M)=\bar k^2+(\bar\mu-\bar\lambda)\bar k+(\bar\mu-\bar k)= \bar\mu v$. As $\theta_m < 0$, we have $\bar k-\theta_m > 0$ and so $\bar k-\theta_M$ and $\bar\mu = \frac{\bar k(\bar k-\bar\lambda-1)}{v-\bar k-1}$ are either both 0, both negative or both positive.
\end{proof}

\medskip \noindent The following lemma will be useful in our discussion. It is precisely Theorem 3.2.1 of \cite{crs} applied to the graph $\Gamma$. 

\begin{lemma} \label{help1}
The largest eigenvalue of $\Gamma$ is at least $\bar k$ with equality if and only if $\Gamma$ is regular.
\end{lemma}

\begin{lemma} \label{help3}
If $\Gamma$ is regular, then $\bar k \geq \bar \lambda +1$.
\end{lemma}
\begin{proof}
For a regular graph $\Gamma$ of degree $k:=\bar k$, we have $\lambda_{xy} + 1 = |\{ y \} \cup (\Gamma_1(x) \cap \Gamma_1(y))| \leq |\Gamma_1(x)|=k$ for every edge $xy$ of $\Gamma$, implying that $\bar\lambda +1 \leq k$. 
\end{proof}

\medskip \noindent If $\Gamma$ is a strongly regular graph, we thus know that $\bar k \geq \bar \lambda +1$ and that $\theta_m$ is the smallest eigenvalue. In fact, we can prove the following.

\begin{lemma} \label{lem10}
Suppose $\bar k \geq \bar \lambda +1$ and let $\theta_{min}$ denote the smallest eigenvalue of $\Gamma$. Then $\theta_{min} \leq \theta_m$ and equality holds if and only if $\Gamma$ is a strongly regular graph.
\end{lemma}
\begin{proof} 
Let $\theta^\ast$ be the largest eigenvalue of $\Gamma$, and put $\Omega' := \Omega_{\theta^\ast}$. By Lemmas \ref{lem9bis} and \ref{help1}, we have $\theta^\ast \geq \bar k \geq \theta_M$. By relying on Lemmas 2.16 and 2.21, we compute
\begin{eqnarray}
 &    & \sum_{\omega \in \Omega} (\omega-\theta_M)^2(\omega-\theta_m) \nonumber \\
 & = & \sum_{\omega \in \Omega} (\omega^2 + (\bar \mu-\bar \lambda) \omega + (\bar \mu-\bar k))(\omega - \theta_M) \nonumber \\
 & = & \sum_{\omega \in \Omega} \Big( \omega^3 + (\bar \mu-\bar \lambda-\theta_M)\omega^2 + (\bar \mu-\bar k-\theta_M(\bar \mu-\bar \lambda)) \omega - \theta_M(\bar \mu-\bar k)  \Big) \nonumber \\
 & = & v\bar k\bar \lambda + (\bar \mu - \bar \lambda - \frac{\bar k-\bar \mu}{\bar k}\bar s)v\bar k - \frac{\bar k-\bar \mu}{\bar k}\bar s(\bar \mu-\bar k)v \nonumber \\
 & = & \bar \mu v\bar k - \frac{\bar \mu(\bar k-\bar \mu)\bar s v}{\bar k}. \nonumber
\end{eqnarray}
This number is equal to
\[ (\bar k-\theta_M)^2(\bar k-\theta_m) = (\bar k^2+(\bar \mu-\bar \lambda) \bar k+(\bar \mu-\bar k))(\bar k-\theta_M) = \bar \mu v(\bar k-\theta_M) = \bar \mu v(\bar k - \frac{\bar k-\bar \mu}{\bar k}\bar s). \]
As the map $\R \to \R; x \mapsto (x-\theta_M)^2(x-\theta_m)$ attains a local maximum for $x=\frac{2\theta_m+\theta_M}{3} < \theta_M$ and a local minimum for $x=\theta_M \leq \bar k \leq \theta^\ast$, we have that
\begin{equation} \label{eq33}
(\theta^\ast-\theta_M)^2(\theta^\ast-\theta_m) \geq (\bar k-\theta_M)^2(\bar k-\theta_m) = \sum_{\omega \in \Omega} (\omega-\theta_M)^2(\omega-\theta_m),
\end{equation}
i.e.
\begin{equation} \label{eq43}
\sum_{\omega \in \Omega'} (\omega-\theta_M)^2(\omega-\theta_m) \leq 0.
\end{equation}
By Lemma 2.15(a)+(b) and the fact that $\Gamma$ is not complete nor a graph without edges, we know that the multiset $\Omega' = \Omega_{\theta^\ast}$ contains at least two distinct elements. Equation (\ref{eq43}) then implies that there exists an $\omega \in \Omega'$ with $\omega \leq \theta_m$, i.e. $\theta_{min} \leq \theta_m$.

Suppose now that $\theta_{min}=\theta_m$. Then $\omega \geq \theta_m$ for every $\omega \in \Omega'$ and so $\sum_{\omega \in \Omega'} (\omega-\theta_M)^2(\omega-\theta_m) \geq 0$. In combination with (\ref{eq43}), this implies that $\omega \in \{ \theta_M,\theta_m \}$ for every $\omega \in \Omega'$. If $\theta^\ast \not= \bar k$, then the inequalities in (\ref{eq33}) and (\ref{eq43}) would be strict, which is impossible. So, $\theta^\ast=\bar k$, and every eigenvalue of $\Gamma$ is equal to $\bar k$, $\theta_m$ or $\theta_M$. The fact that $\theta^\ast=\bar k$ implies by Lemma \ref{help1} that $\Gamma$ is regular with valency $k := \bar k$.

If $\Gamma$ is connected, then $\Gamma$ is strongly regular by Lemma 2.15. If $\Gamma$ is not connected, then each connected component of $\Gamma$ is regular of degree $k$ and as each $\omega \in \Omega_{\theta^\ast}=\Omega_k$ belongs to $\{ \theta_M,\theta_m \}$ we then know that $\Gamma$ has at most two eigenvalues, implying by Lemma 2.15(a)+(b) that $\Gamma$ is a disjoint union of at least two complete graphs of the same order. In this case, $\Gamma$ is thus also a strongly regular graph.
\end{proof}

\medskip \noindent The following is now a consequence of Lemmas \ref{help3} and \ref{lem10}.

\begin{corollary} \label{co11bis}
If $\Gamma$ is regular with smallest eigenvalue $\theta_{min}$, then $\theta_{min} \leq \theta_m$ with equality if and only if $\Gamma$ is strongly regular.
\end{corollary}

\medskip \noindent If $\Gamma$ is a strongly regular graph with valency $k$, then we know that $\theta_M$ is the largest eigenvalue in $\Omega_k$. In fact, the following can be proved.

\begin{lemma} \label{lem12}
Suppose $\Gamma$ is regular with valency $k := \bar k$ and let $\theta_{Max}$ denote the largest eigenvalue of $\Gamma$ in $\Omega' := \Omega_k$. Then $\theta_{Max} \geq \theta_M$ with equality if and only if $\Gamma$ is strongly regular.
\end{lemma}
\begin{proof}
By relying on Lemmas 2.16 and 2.21, we compute
\begin{eqnarray}
 &    & \sum_{\omega \in \Omega} (\omega-\theta_M)(\omega-\theta_m)^2 \nonumber \\
 & = & \sum_{\omega \in \Omega} (\omega^2 + (\bar \mu-\bar \lambda) \omega + (\bar \mu-\bar k))(\omega - \theta_m) \nonumber \\
 & = & \sum_{\omega \in \Omega} \Big( \omega^3 + (\bar \mu-\bar \lambda-\theta_m)\omega^2 + (\bar \mu-\bar k-\theta_m(\bar \mu-\bar \lambda))\omega-\theta_m(\bar \mu-\bar k)  \Big) \nonumber \\
 & = & v\bar k\bar \lambda + (\bar \mu - \bar \lambda +\frac{\bar k}{\bar s})v\bar k +\frac{\bar k}{\bar s}(\bar \mu-\bar k)v \nonumber \\
 & = & \bar \mu v\bar k + \frac{\bar \mu v\bar k}{\bar s}. \nonumber
\end{eqnarray}
This number is equal to
\[ (\bar k-\theta_M)(\bar k-\theta_m)^2 = (\bar k^2+(\bar \mu-\bar \lambda) \bar k+(\bar \mu-\bar k))(\bar k-\theta_m) = \bar \mu v(\bar k-\theta_m) = \bar \mu v(\bar k+\frac{\bar k}{\bar s}). \]
We thus have
\begin{equation} \label{eq64}
\sum_{\omega \in \Omega'} (\omega-\theta_M)(\omega-\theta_m)^2 = 0.
\end{equation}
By Lemma 2.15(a)+(b) and the fact that $\Gamma$ is not complete nor a graph without edges, we know that the multiset $\Omega' = \Omega_k$ contains at least two distinct elements. Equation (\ref{eq64}) then implies that there exists an $\omega \in \Omega'$ with $\omega \geq \theta_M$, i.e. $\theta_{Max} \geq \theta_M$.

Suppose now that $\theta_{Max} = \theta_M$. Then $\omega \leq \theta_M$ for every $\omega \in \Omega'$. In combination with (\ref{eq64}), this implies that $\omega \in \{ \theta_M,\theta_m \}$ for every $\omega \in \Omega'$.

If $\Gamma$ is connected, then $\Gamma$ is strongly regular by Lemma 2.15. If $\Gamma$ is not connected, then each connected component of $\Gamma$ is regular of degree $k$ and as each $\omega \in  \Omega' = \Omega_k$ belongs to $\{ \theta_M,\theta_m \}$, we then know that $\Gamma$ has at most two eigenvalues, implying by Lemma 2.15(a)+(b) that $\Gamma$ is a disjoint union of at least two complete graphs of the same order. In this case, $\Gamma$ is thus also a strongly regular graph.
\end{proof}

\medskip \noindent Theorem \ref{Neumaier4mainthm} is now implied by Corollary \ref{co11bis} and Lemma \ref{lem12}.

\subsection{Characterizations using the Hoffman's ratio bound }\label{Sec:Neumaier3file}

Delsarte [\cite{Delsarte}, page 31] showed that if $C$ is a clique in a strongly regular graph $\Gamma$ with parameters $(v,k,\lambda,\mu)$, then  
\begin{equation} \label{Delsarte1}
|C| \leq 1 - \frac{k}{\theta_{\min}},
\end{equation}
with $\theta_{\min}$ the smallest eigenvalue of $\Gamma$.  A coclique $C'$ in $\Gamma$ can be regarded as a clique in the complementary graph $\overline{\Gamma}$ of $\Gamma$ and so we have
\begin{equation} \label{Delsarte2}
|C'| \leq 1 - \frac{v-k-1}{\theta_{Cmin}} = \frac{v}{1-\frac{k}{\theta_{\min}}},
\end{equation}
with $\theta_{Cmin}$ the smallest eigenvalue of $\overline{\Gamma}$. We call a (co)clique that meets the \emph{Delsarte bound} in (\ref{Delsarte1}) or (\ref{Delsarte2}) a  \emph{Delsarte-(co)clique}. Note that many people call them \emph{Hoffman-(co)cliques}. The bound for strongly regular graphs, however, was first given by Delsarte. Hoffman later generalized it to arbitrary regular graphs, and the bound is often called the {\em ratio bound}. Specifically, Hoffman showed that
\begin{equation} \label{Hoffman1}
|C'| \leq  v \Big( 1-\frac{k}{\theta_{\min}} \Big)^{-1}
\end{equation}
for any coclique $C'$ in a regular graph with valency $k \geq 1$ on $v$ vertices having smallest eigenvalue $\theta_{\min}$, with equality if and only if every vertex outside $C$ has the same number of neighbours in $C$. A similar characterization for the equality case holds for the inequality (\ref{Delsarte1}). For proofs and more background information on these bounds, see [\cite{bcn}, Propositions 1.3.2 and 4.4.6].

\begin{theorem}\label{Thmhoffman}
Let $\Gamma = (V,E)$ be a connected, non-complete edge-regular graph. Then the following statements are equivalent:
\begin{enumerate}
\item[$(1)$] $\Gamma$ is a strongly regular Neumaier graph;
\item[$(2)$] there exists a clique $C$ in $\Gamma$ such that $|C|$ attains the Hoffman coclique bound in the complementary graph $\overline{\Gamma}$.
\end{enumerate}
\end{theorem}
\begin{proof}
Let $v$, $k$ and $\lambda$ be the parameters of $\Gamma$ as an edge-regular graph, and define $s$ as in the beginning of Section \ref{Sec:Neumaier1file}. 

\medskip \noindent $(2) \Rightarrow (1).$ In case $\Gamma$ is a complete multipartite graph, it must be a strongly regular Neumaier graph, see e.g. Section \ref{Sec:Neumaier1file}. So, for this part of the proof, we may assume that $\Gamma$ is not a complete multipartite graph. Suppose there exists a clique $C$ in $\Gamma$ such that $|C|$ attains the Hoffman coclique bound in $\bar{\Gamma}$. Then
\[  |C| = \frac{v}{1 + \frac{v-k-1}{-\theta_{Cmin}}} = \frac{v}{1 + \frac{v-k-1}{\theta_{Max2}+1}},   \]
where $\theta_{Cmin}$ is the smallest eigenvalue of $\overline{\Gamma}$. Note that $-\theta_{Cmin}=\theta_{Max2}+1$, where $\theta_{Max2}$ is the second largest eigenvalue of $\Gamma$, see e.g. Theorem 2.6 of \cite{cds}. We also know that every vertex of $V \setminus C$ is $\overline{\Gamma}$-adjacent to a constant number of vertices of $C$, or equivalently, $\Gamma$-adjacent to a constant number $e$ of vertices of $C$. As $\Gamma$ is connected and non-complete, $e > 0$ and so $C$ is a regular clique. By Proposition \ref{prop2}, we then know that $|C|=s+1$ and $e=\frac{(s+1)(k-s)}{v-(s+1)}$.  The quotient matrix $\begin{bmatrix} s & k-s \\ e & k-e\end{bmatrix}$ of the equitable partition $\{ C,V \setminus C \}$ has then $k$ and $s-e$ as eigenvalues. So, $k$ and $s-e$ are also eigenvalues of $\Gamma$ and $\theta_{Max2} \geq s-e$. As $s-e+1 = \frac{(s+1)(v-(k+1))}{v-(s+1)}$, we have
\[ s+1 = |C| = \frac{v}{1+\frac{v-k-1}{\theta_{Max2}+1}} \geq \frac{v}{1+\frac{v-k-1}{s-e+1}} = \frac{v}{1+\frac{v-(s+1)}{s+1}} = \frac{v(s+1)}{s+1+v-(s+1)}=s+1, \]
implying that $\theta_{Max2}=s-e$. So, the second largest eigenvalue of $\Gamma$ is $s-e$, implying by Theorem \ref{Neumaier4mainthm}(2) and Lemma \ref{lem8bis} that $\Gamma$ is strongly regular. As $\Gamma$ has a regular clique, it is a Neumaier graph.

\medskip \noindent $(1) \Rightarrow (2).$ Suppose $\Gamma$ is a strongly regular Neumaier graph. If $\Gamma$ is not a complete multipartite graph, then by Lemmas \ref{lem7bis} and \ref{lem8bis}, the eigenvalues of $\Gamma$ are $k$, $s-e$ and $-\frac{k}{s}$. In fact, by direct verification one can see that this remains true for complete multipartite graphs. So, $\theta_{Max2}=s-e$. The upper bound in the Hoffman coclique bound is then equal to
\[  \frac{v}{1+\frac{v-k-1}{\theta_{Max2}+1}}=\frac{v}{1+\frac{v-k-1}{s-e+1}}=\frac{v}{1+\frac{v-(s+1)}{s+1}}=\frac{v(s+1)}{s+1+v-(s+1)}=s+1.   \]
Every regular clique in $\Gamma$ has order $s+1$, and so there exists a coclique in $\overline{\Gamma}$ attaining the Hoffman coclique bound.
\end{proof}

\begin{theorem}\label{thmsrgeigenvalue}
The following are equivalent for a Neumaier graph $\Gamma$.
\begin{enumerate}
\item[$(1)$] $\Gamma$ is strongly regular.
\item[$(2)$] The Delsarte clique bound holds for $\Gamma$, i.e. $|C| \leq 1 - \frac{k}{\theta_{min}}$ for every clique $C$, with $k$ the valency of $\Gamma$ and $\theta_{min}$ the smallest eigenvalue of $\Gamma$.
\end{enumerate}
\end{theorem}
\begin{proof}
Let $v$, $k$ and $\lambda$ be the parameters of $\Gamma$ as an edge-regular graph, and define $s$ and $e$ as in the beginning of Section \ref{Sec:Neumaier1file}. If $C$ is a clique of $\Gamma$, then by Propositions \ref{prop2} and \ref{prop3} we know that $|C| \leq s+1$, with equality if and only if $C$ is a regular clique.

\medskip \noindent $(1) \Rightarrow (2).$ Suppose $\Gamma$ is strongly regular. Similarly as in Theorem \ref{Thmhoffman}, we then know that $k$, $s-e$ and $-\frac{k}{s}$ are the eigenvalues of $\Gamma$. So, $1-\frac{k}{\theta_{min}}=s+1$. As $|C| \leq s+1$ for every clique $C$, we indeed see that the Delsarte clique bound holds in $\Gamma$.

\medskip \noindent $(2) \Rightarrow (1).$ Suppose the Delsarte clique bound holds in $\Gamma$. If $C$ is a regular clique in $\Gamma$, then $s+1 = |C| \leq 1 - \frac{k}{\theta_{min}}$, implying that $\theta_{min} \geq -\frac{k}{s}$. Then by Corollary \ref{co11bis} it follows that $\Gamma$ is strongly regular.
\end{proof}

\subsection{A characterization using $t$-walk regularity}

Recall that a \emph{$t$-walk-regular graph} is a graph $\Gamma$ for which the number of walks of a given length between two vertices depends only on the distance between these two vertices, as long as this distance is at most $t$. Such graphs generalize distance-regular graphs and $t$-arc-transitive graphs. Every $t$-walk-regular graph is also regular.

Let $\Gamma$ be a connected regular graph with distinct eigenvalues $k= \theta_0,\theta_1,\ldots,\theta_d$ and adjacency matrix $A$. Let $R_i := \frac{\prod_{j: j \neq i} (A-\theta_j I)}{\prod_{j:j \neq i} \theta_i-\theta_j}$ with $i \in \{ 0,1,\ldots,d \}$ be the \emph{primitive idempotent} of $\Gamma$ corresponding to the eigenvalue $\theta_i$. We have $AR_i = \theta_i R_i$. Let $A_i$ be the $i$-adjacency matrix, i.e. $(A_i)_{xy} = 1$ if $d(x, y)=i$ and 0 otherwise.

By Theorem 3.1 of \cite{cfg}, we know that $\Gamma$ is $t$-walk-regular if and only if for every $i \in \{ 0,1,\ldots,t \}$ and every $j \in \{ 0,1,\ldots,d \}$, there exists an $\alpha_{ij} \in \R$ such that $A_i\circ R_j = \alpha_{ij} A_i$, where $\circ$ denotes the entrywise product. From now on we assume that $\Gamma$ is $1$-walk-regular. For every $i \in \{ 0,1,\ldots,d \}$, we define $\epsilon_i := (R_i)_{xx}$ where $x$ is any vertex of $\Gamma$, and put $\tilde{R_i} := \frac{1}{\epsilon_i} R_i$. It is easy to see that $(\tilde{R_i})_{xy} = \theta_i/k$ for adjacent vertices $x$ and $y$.

Neumaier [\cite{Neumaier1981}, Corollary 2.4] showed that a vertex- and edge-transitive graph that has a regular clique must be strongly regular. Since vertex- and edge-regularity implies 1-walk-regularity, the next theorem provides an analogous of Neumaier's result, but it requires a weaker assumption.

\begin{theorem}\label{thm:Hefeidiscussions1}
Let $\Gamma = (V,E)$ be a $1$-walk-regular graph with a regular clique. Then $\Gamma$ is a strongly regular graph.
\end{theorem}
\begin{proof}
Let $k$ be the valency of $\Gamma$. Let $C$ be a regular clique of size $s+1$ and suppose every vertex outside $C$ has $e$ neighbours in $C$. As $e \geq 1$, the graph $\Gamma$ is connected. Consider the partition ${\cal{P}}=\{C,V\setminus C\}$, which has quotient matrix $B=\begin{bmatrix}
s & k-s \\
e & k-e
\end{bmatrix}$.
The eigenvalues $k$ and $s-e < k$ of $B$ are also eigenvalues of the adjacency matrix $A$ of $\Gamma$. If $[v_1 \, v_2]^T$ is an eigenvector of $B$ corresponding to $s-e$, then $v_1 \not= v_2$ and the column matrix $\omega$ of size $|V|$ with $\omega_x = v_1$ if $x \in C$ and $\omega_x=v_2$ otherwise is an eigenvector of $A$ corresponding to the eigenvector $s-e$. Note that $j = [1\, 1 \, \cdots \, 1]^T$ is an eigenvector of $A$ corresponding to $k$. The characteristic vector $\chi_C$ of $C$ is a linear combination of $\omega$ and $j$. 

Assume now that $\theta_i \not\in \{ k,s-e \}$ is another eigenvalue, with idempotent $R_i$, and consider the matrix $\tilde{R_i}$ as defined above. Any column of $\tilde{R_i}$ is orthogonal to $\omega$ and $j$ and hence also to $\chi_C$. Considering a column of $\tilde{R_i}$ corresponding to a vertex of $C$, we then find that $1+su=0$, where $u=\frac{\theta_i}{k}$. The eigenvalue $\theta_i$ of $\Gamma$ is thus uniquely determined. It follows that $\Gamma$ has at most three distinct eigenvalues. As $\Gamma$ is also connected and regular, it must be a strongly regular graph.
\end{proof}

\medskip Theorem \ref{thm:Hefeidiscussions1} has the following immediate corollary.

\begin{corollary}
It is not possible to construct a Neumaier graph as a relation graph of a symmetric association scheme that does not come from a strongly regular graph.
\end{corollary}

\medskip It is worth noting that there are vertex-transitive Neumaier graphs that are not strongly regular graphs, see for instance the construction provided by Greaves and Koolen \cite{GK2018} which uses Cayley graphs. Note that vertex-transitivity implies 0-walk-regular.

It is also important to note that a vertex-transitive Neumaier graph
has diameter 2, and it is not known whether this is still true under the weaker condition
 of walk-regularity. Examples of Neumaier graphs with diameter $3$ are known \cite{Sergeyprivate}.

\subsection{A characterization using the smallest eigenvalue being $-2$}\label{smallestev-2}

In [Theorem 3.12.4, \cite{bcn}] it is shown that if a connected graph is edge-regular with smallest eigenvalue $-2$, then it is strongly regular or the line graph of a triangle-free regular graph.
	
\begin{proposition}\label{smallesteigenvalue-2}
Suppose the line graph $L(\Gamma)$ of a graph $\Gamma$ is a Neumaier graph. Then $L(\Gamma) $ is one of the following:
\begin{enumerate}
\item[$(1)$] an $(s+1) \times (s+1)$ rook graph for some $s \in \N \setminus \{ 0 \}$;
\item[$(2)$] a Johnson graph $J(s+2,2)$ for some $s \in \N \setminus \{ 0,1 \}$;
\item[$(3)$] the octahedral graph.
\end{enumerate}
In particular, $L(\Gamma)$ is a strongly regular graph.
\end{proposition}
\begin{proof}
As adding isolated vertices to a graph does not alter its line graph, we may without loss of generality assume that $\Gamma$ does not contain isolated vertices. As $L(\Gamma)$ is connected, the graph $\Gamma$ must therefore be connected as well. We denote by $k$ the valency of $L(\Gamma)$ and by $\lambda$ the constant number of triangles through a given edge of $L(\Gamma)$. We also assume that $C$ is a regular clique of size $s+1$ of $L(\Gamma)$ and that every vertex of $L(\Gamma)$ not in $C$ is adjacent to exactly $e > 0$ vertices of $C$. As $C$ is a maximal clique, it consists either of all the edges in $\Gamma$ containing a given vertex or of all three edges in a triangle of $\Gamma$. 
  
\medskip  Suppose first that $C$ consists of all the edges in $\Gamma$ through a vertex $v$. Every edge of $\Gamma$ not incident with $v$ is adjacent to at most two edges in $C$, implying that $e \in \{ 1,2 \}$.

Assume $e=1$. Then $\Gamma(v)$ cannot contain edges. If $v'$ is a vertex not belonging to $\{ v \} \cup \Gamma(v)$, then the condition $e=1$ implies that every edge through $v'$ is of the form $v'w$ with $w$ one of the $s+1$ vertices of $\Gamma(v)$. The degree $k$ of the edges $vw$ and $v'w$ in $L(\Gamma)$ is equal to $\deg(v)+\deg(w)-2=\deg(v')+\deg(w)-2$, implying that $\deg(v')=\deg(v)=s+1$. This implies that $\Gamma \cong K_{s+1,n}$ for some $n \in \N \setminus \{ 0,1 \}$. Edge-regularity in $L(\Gamma)$ implies that $n=s+1$ and that $L(\Gamma)$ is the $(s+1) \times (s+1)$ rook graph.
 
Assume $e=2$. Then all edges not incident with $v$ are contained in $\Gamma(v)$, and so $\{ v \} \cup \Gamma(v)$ is the whole vertex set of $\Gamma$. If $b=w_1w_2$ is an edge not incident with $v$, then the edge-regularity condition applied to respectively $\{ vw_1,vw_2 \}$ and $\{ vw_1,w_1w_2 \}$ gives $\lambda = \deg(v) - 2 + 1 = \deg(w_1)-2+1$, implying that $\deg(w_1)=\deg(v)=s+1$. Repeating this argument for other edges not incident with $v$, we see that all vertices of $\Gamma$ have degree $s+1$. So, $\Gamma \cong K_{s+2}$ and $L(\Gamma) \cong J(s+2,2)$. As $L(\Gamma)$ is not complete, we have $s \geq 2$.  
 
\medskip Suppose next that $C = \{ v_1v_2,v_1v_3,v_2v_3 \}$ where $\Delta = \{ v_1,v_2,v_3 \}$ is a triangle of $\Gamma$. The valency $k$ of $L(\Gamma)$ is equal to $\deg(v_1)+\deg(v_2)-2=\deg(v_1)+\deg(v_3)-2=\deg(v_2)+\deg(v_3)-2$, implying that $\deg(v_1)$, $\deg(v_2)$ and $\deg(v_3)$ are equal, say to $l$, and that $k=2l-2$. Since $\Gamma$ is connected and has edges outside $C$, we have $l \geq 3$. Note that the set of $l \geq 3$ edges through $v_1$ is a clique of $L(\Gamma)$. As the maximal size of a clique of $L(\Gamma)$ is equal to $|C|=3$, we have $l=3$ and $k=4$. Every vertex $v_i \in \{ v_1,v_2,v_3 \}$ is therefore incident with a unique edge $e_i$ not contained in $\Delta$. We see that $e_i$ and hence all vertices $b \not\in C$ are adjacent with precisely $e=2$ edges of $C$. So, $e_1$, $e_2$ and $e_3$ are the only edges not in $C$. As the degree of $e_1$ in $L(\Gamma)$ equals $k=4$, the edges $e_1$, $e_2$ and $e_3$ are mutually adjacent, implying that $\Gamma$ is the tetrahedral graph and $L(\Gamma)$ is the octahedral graph.
\end{proof}	
	
Proposition \ref{smallesteigenvalue-2} and [\cite{bcn}, Theorem 3.12.4] imply the following straightforward consequence.

\begin{corollary}
Every Neumaier graph with smallest eigenvalue $-2$ is strongly regular.
\end{corollary}

\section{Neumaier graphs with four eigenvalues}\label{Neumaier4distinctev}

Regular graphs with precisely four eigenvalues have already been studied in the literature, see e.g. \cite{v1995}. In this section, we prove that no such graph can be a Neumaier graph.

\begin{theorem}
There are no Neumaier graphs with exactly four distinct eigenvalues.
\end{theorem}
\begin{proof}
Suppose to the contrary that $\Gamma$ is a Neumaier graph with exactly four eigenvalues $k$, $\theta_1$, $\theta$ and $\theta_2$, with $k$ the valency of $\Gamma$ and $k > \theta_1 > \theta > \theta_2$. Suppose $C$ is a regular clique of size $s+1$ and every vertex outside $C$ is adjacent to precisely $e > 0$ vertices of $C$. Above, we have already seen that $e = \frac{(s+1)(k-s)}{v-s-1}$ with $v$ the number of vertices of $\Gamma$, i.e.
\begin{equation}\label{countingtvertices}
v=\frac{(s+1)(k-s+e)}{e}.
\end{equation}
Counting triangles having exactly one of its edges in $C$, we find 
\begin{equation}\label{countingtriangles}
{s+1\choose 2}(\lambda-(s-1))=(v-s-1){e\choose 2},
\end{equation}
with $\lambda$ denoting the constant number of triangles containing a given edge.

The equitable partition $\{ C,V\setminus C \}$ has 
quotient matrix
$
\begin{bmatrix}
s & k-s \\
e & k-e
\end{bmatrix},
$
whose eigenvalues are $k$ and $s-e$. As these are also eigenvalues of $\Gamma$, we have $s-e \in \{ \theta,\theta_1,\theta_2 \}$. In fact, we can prove that $s-e=\theta$. As $C$ cannot be properly contained in another clique, we have $e \leq s$. So, $s-e \geq 0$ is distinct from the smallest eigenvalue $\theta_2$ of $\Gamma$ which is negative. As $\Gamma$ has exactly four eigenvalues, it is not a strongly regular graph, implying by Theorem \ref{Neumaier4mainthm} and Lemma \ref{lem8bis} that $\theta_1 > s-e$. Summarizing we thus have
\[  k > \theta_1 > \theta=s-e \geq 0 > \theta_2.  \]
As $s=\theta+e$, equations (\ref{countingtvertices}) and (\ref{countingtriangles}) become
\[  v=\frac{(\theta+e+1)(k-\theta)}{e},\qquad  \lambda = \theta+e-1+\frac{(k-\theta-e)(e-1)}{\theta+e}. \] 
As $|C|=s+1=\theta+e+1$ and $\Gamma$ is not strongly regular, Theorem \ref{thmsrgeigenvalue} implies that $\theta_2 < - \frac{k}{\theta+e}$, or equivalently
\[  k+e\theta_2+\theta_2 \theta < 0.  \]
The eigenvalues of the adjacency matrix $A$ of $\Gamma$ distinct from $k$ are $\theta$, $\theta_1$ and $\theta_2$. As $\Gamma$ is connected and regular with valency $k$, we thus have
\begin{equation}\label{eq5}
(A-\theta I)(A-\theta_1 I)(A-\theta_2 I)=\alpha J
\end{equation}
for some $\alpha \in \R$, where $J$ denotes the $v \times v$ matrix with all entries equal to 1. Equation (\ref{eq5}) can be rewritten as 
\begin{equation}\label{eq6}
A^3+\beta_2A^2+\beta_1 A +\beta_0 I =\alpha J,
\end{equation}
where $\beta_2=-(\theta+\theta_1+\theta_2)$, $\beta_1 = \theta \theta_1 + \theta \theta_2 + \theta_1 \theta_2$ and $\beta_0=-\theta\theta_1\theta_2$. If we calculate the diagonal entries of the matrices occurring at both sides of (\ref{eq6}), we find that 
\[ \alpha = k\lambda+\beta_2k+\beta_0=k\lambda+(-\theta-\theta_1-\theta_2)k-\theta\theta_1\theta_2, \]
as there are $k$ and $k \lambda$ closed walks of respective lengths 2 and 3 starting and ending in a given vertex of $\Gamma$. On the other hand, after multiplying both sides of (\ref{eq5}) by $j = [1 \, 1\, \cdots \, 1]^T$, we find
\[  (k-\theta)(k-\theta_1)(k-\theta_2) = \alpha v = (k\lambda+(-\theta-\theta_1-\theta_2)k-\theta\theta_1\theta_2) \frac{(\theta+e+1)(k-\theta)}{e},   \]
i.e. 
$$(k-\theta_1)(k-\theta_2)=(k\lambda+(-\theta-\theta_1-\theta_2)k-\theta\theta_1\theta_2)\frac{(\theta+e+1)}{e}.$$
The latter equation implies that
\[ \Big( \theta_2 e+k\theta+k+\theta_2 \theta^2+e\theta_2\theta+\theta_2\theta \Big) \theta_1 =  -k^2e+k\theta_2 e+k (\theta+e+1)(\lambda-\theta-\theta_2).  \]
As $\theta_2 e+k\theta+k+\theta_2 \theta^2+e\theta_2\theta+\theta_2\theta = (\theta+1)(k+e\theta_2+\theta_2\theta) \not= 0$, we thus obtain
\begin{align*}
\theta_1 & =\frac{-k^2e+k\theta_2 e+k (\theta+e+1)(e-1+\frac{(k-\theta-e)(e-1)}{\theta+e}-\theta_2)}{(\theta+1)(k+e\theta_2+\theta_2\theta)}\\
    & =\frac{-k(k\theta+k+\theta_2\theta+\theta_2\theta^2+e\theta_2+e\theta_2\theta)}{(\theta+1)(k+e\theta_2+\theta_2\theta)(e+\theta)}\\
    & =\frac{-k}{e+\theta}.
\end{align*}
This is impossible as $\theta_1 > 0 > -\frac{k}{e+\theta}$.
\end{proof}

\section*{Acknowledgement}
The research of A. Abiad is partially supported by the BOF-UGent (Special Research Fund of Ghent University). The research of J. D’haeseleer is supported by the FWO (Research Foundation Flanders). J.H. Koolen is partially supported by the National Natural Science Foundation of China (No. $11471009$ and No. $11671376$) and Anhui Initiative of Quantum Information Technologies (No. AHY $150000$).

\end{document}